\documentclass[12pt,twoside]{amsart}
\usepackage{amsmath,latexsym,amssymb,amscd}
\usepackage{setspace}
\usepackage{fancyvrb}
\usepackage{tikz}

\usepackage[active]{srcltx}
\nonstopmode

\DefineVerbatimEnvironment{code}{Verbatim}{fontsize=\small}
\DefineVerbatimEnvironment{example}{Verbatim}{fontsize=\small}

\theoremstyle{definition}
\newtheorem{thm}{Theorem}[section]
\newtheorem{lem}[thm]{Lemma}
\newtheorem{rmk}[thm]{Remark}
\newtheorem{prop}[thm]{Proposition}

\newtheorem{eg}[thm]{Example}
\newtheorem{cor}[thm]{Corollary}
\newtheorem{notation}[thm]{Notation}
\newtheorem{dfn}[thm]{Definition}

\newcommand{\Z}{\mathbb{Z}}

\newcommand{\PP}{\mathbb{P}}

\newcommand{\cO}{\mathcal{O}}

\newcommand{\pdk}{\pi_{d-{\rm lin}; k}}
\newcommand{\qdk}{\pi_{0, d-{\rm lin}; k_S}}
\newcommand \pp[1] {^{\langle #1 \rangle}}

\DeclareMathOperator{\gin}{gin}

\DeclareMathOperator{\Ext}{Ext}
\DeclareMathOperator{\Hom}{Hom}
\DeclareMathOperator{\reg}{reg}
\DeclareMathOperator{\Tor}{Tor}

\newcommand{\ffi}{\varphi}
\newcommand{\al}{\alpha}
\newcommand{\be}{\beta}

\numberwithin{equation}{section}

\title[Combinatorial Interpretations of Decompositions]{Combinatorial Interpretations of some Boij-S\"oderberg Decompositions}
\author[U.\ Nagel]{Uwe Nagel}
\author[S.\ Sturgeon]{Stephen Sturgeon}
\address{Department of Mathematics, University of Kentucky, 718 Patterson Office Tower, Lexington, KY 40506-0027, USA}
\email{uwe.nagel@uky.edu, stephen.sturgeon@uky.edu}

\thanks{This work was partially supported by a grant from the
Simons Foundation (\#208869 to Uwe Nagel) and by the National Security Agency under Grant Number H98230-12-1-0247. \\
The second author also would like to thank MSRI for organizing and funding the summer workshop on commutative algebra in 2011 as well as Daniel Erman for his inspiring lectures. }

\begin{document}

\begin{abstract}
    Boij-S\"oderberg theory shows that the Betti table of a graded module can be written as a liner combination of pure diagrams with integer coefficients. Using Ferrers hypergraphs and simplicial polytopes, we provide interpretations of these coefficients for ideals with a $d$-linear resolution, their quotient rings, and for Gorenstein rings whose resolution has essentially at most two linear strands.
    We also establish a structural result on the decomposition in the case of quasi-Gorenstein modules.
\end{abstract}

\maketitle


\section{{Introduction}}

Boij-S\"oderberg theory  classifies all Betti tables of graded modules over a polynomial ring $R$  up to a rational multiple. This is achieved by writing the Betti table of such a module as a unique linear combination of pure diagrams whose coefficients are positive integers (see Section \ref{sec:prelim} for details). The purpose of this paper is to demonstrate combinatorial significance of these coefficients in a few cases. Each of these cases is related to ideals that are derived from some combinatorial objects. Moreover, we show that the self-duality of the minimal free resolution of a quasi-Gorenstein module is reflected in its Boij-S\"oderberg decomposition. This includes all standard graded Gorenstein algebras.

In Section \ref{sec:linear-res} we first consider the Boij-S\"oderberg decompositions of ideals with a $d$-linear resolution. The Ferrers ideals associated to $d$-uniform Ferrers hypergraphs provide examples of such ideals. Their resolutions are well-understood thanks to results in \cite{CN1}, \cite{CN2}, and \cite{NR1}. We show that each Betti table of an ideal with a $d$-linear resolution corresponds to the Betti table of a suitable Ferrers hypergraph. This allows us to give a combinatorial interpretation of the Boij-S\"oderberg coefficients. In particular, they must form an $O$-sequence. This result (see Theorem \ref{thm:charact-lin-resolution}) provides   a characterization of the Betti numbers of ideals with a $d$-linear resolution that complements the recent characterization obtained in \cite{HSV}.

Then we consider the Boij-S\"oderberg decomposition of certain quotient rings $R/I$.  The Betti tables of the ideal $I$ and $R/I$ are closely related. However, their Boij-S\"oderberg decompositions are very different in general and the precise relationship is not known. In the case where $I$ has a $d$-linear resolution, we obtain an interpretation of the Boij-S\"oderberg coefficients of $R/I$ (see Theorem \ref{thm:BS-quotient-linear}). Again it relies on a suitable Ferrers hypergraph though the coefficients are extracted by counting different subsets in the hypergraph this time.

Quasi-Gorenstein modules were introduced in \cite{mod-liai}. They are important in the liaison theory of modules. Gorenstein rings are examples of cyclic such modules. Quasi-Gorenstein modules have a self-dual minimal free resolution.

In Section \ref{sec:quasi-Gor} we consider the  Boij-S\"oderberg decomposition of such modules. We show that their Betti table can be rewritten as a linear combination of self-dual diagrams, where each summand is the sum of at most two pure diagrams (see Theorem \ref{thm:BS-self-dual}). Specific instances of such decompositions are derived in Section \ref{sec:Gor-rings}. We consider the Betti tables of Gorenstein rings with few linear strands. Such Betti tables arise naturally. In particular, they can be obtained from the resolutions of Stanley-Reisner rings corresponding to boundary complexes of simplicial polytopes. The Boij-S\"oderberg decomposition are described in Theorem \ref{thm:Gor-ideals}. The coefficients admit a very transparent interpretation in the case of stacked polytopes (see Corollary \ref{cor:stacked-polytope}).

We review basic facts on Boij-S\"oderberg decompositions and Ferrers hypergraphs in Section \ref{sec:prelim}. Furthermore, given any strongly stable monomial ideal $I$ whose generators have degree $d$,  an explicit construction of a $d$-uniform Ferrers hypergraph with the same graded Betti numbers as $I$ is provided in Remarks \ref{rem:polarization} and \ref{rem:sqr-stable-Ferrers}.


\section{{Boij-S\"oderberg decomposition, $O$-sequences, and Ferrers hypergraphs}}
  \label{sec:prelim}

We recall some results and concepts that are needed in subsequent sections.

We work over a polynomial ring $R = K[x_1,\ldots,x_n]$ in $n$ variables, where $K$ is any field. All modules  are assumed to be graded finitely generated $R$-modules. We denote the graded Betti numbers of $M$ (as an $R$-module) by
\[
\be_{i, j} (M) = \dim_K [\Tor^R_i (M, K)]_j.
\]
The numerical information of the minimal free resolution of $M$ is captured in  the Betti table  $\be (M) = (\be_{i, j} (M))$ of $M$.

\begin{dfn}
  \label{def:pure-diag}
Given an increasing sequence of integers $\sigma = (d_0, d_1,\ldots,d_s)$, where $0 \le s \le n$, we denote by $\pi_{\sigma}$ the matrix with entries $\be_{i, j}$, where
\begin{equation*}
  \label{eq:Betti-pure-table}
  \be_{i, j} = \begin{cases}
    {\displaystyle \prod_{j=0,j\neq i}^n \frac{1}{|d_i-d_j|}} & \text{if } j = d_i \\[1ex]
    0 & \text{otherwise}.
  \end{cases}
\end{equation*}
It is called the {\em pure diagram} to the {\em degree sequence $\sigma$}.
\end{dfn}

Note that this is the convention used in \cite{BS2}, which differs from the original proposal in \cite{BS1}. For example:
\[
\pi_{(0,2,3,5)}=\begin{array}{c|cccc}
\beta_{i,j} & 0 & 1 & 2 & 3\\
\hline
0 & \frac{1}{30} & 0 & 0 & 0\\
1 & 0 & \frac{1}{6} & \frac{1}{6} & 0\\
2 & 0 & 0 & 0 & \frac{1}{30}
\end{array}
\]

Every pure diagram is a rational multiple of the Betti table of a Cohen-Macaulay module. Due to the seminal results by Eisenbud and Schreyer in \cite{ES}, much more is true. Define a partial order on the set of pure diagrams  by setting $\pi_{\sigma} \le \pi_{\tau}$, where $\sigma = (d_0, d_1,\ldots,d_s)$ and $\tau = (d_0', d_1',\ldots,d_t')$ are degree sequences,   if $s \ge t$ and $d_i \le d_i'$ for all $i = 0,1,\ldots,t$. Boij-S\"oderberg theory as developed by Boij and S\"oderberg in \cite{BS1}, \cite{BS2} and Eisenbud and Schreyer in \cite{ES} classifies all Betti tables of graded $R$-modules up to a rational multiple. More precisely, one has (see \cite[Theorem 2]{BS2}):

\begin{thm}
  \label{thm:BS-decomposition}
For every graded, finitely generated $R$-module $M$, there are unique pure diagrams $\pi_{\sigma_1} < \pi_{\sigma_2} < \cdots < \pi_{\sigma_t}$ and positive integers $a_1,\ldots,a_t$ such that
\begin{equation}
   \label{eq:BS-decomp}
\be (M) = \sum_{i=1}^t a_i \pi_{\sigma_i},
\end{equation}
\end{thm}

We call the right-hand side in Equation \eqref{eq:BS-decomp}  the
{\em Boij-S\"oderberg decomposition} of the Betti table of $M$, the pure diagrams $\pi_{\sigma_i}$ its {\em summands} and the integers $a_i$ the {\em Boij-S\"oderberg coefficients} of $M$.
\smallskip

Next, we recall Macaulay's characterization of Hilbert functions of graded $K$-algebras. Given positive integers $b$ and $d$, there are unique integers $m_d >
m_{d-1} > m_s \geq s \geq 1$ such that
\[
b = \binom{m_d}{d} + \binom{m_{d-1}}{d-1} + \ldots + \binom{m_s}{s}.
\]
Then define
\[
b\pp{d} :=  \binom{m_d + 1}{d + 1} + \binom{m_{d-1} + 1}{d } + \ldots +
\binom{m_s + 1}{s + 1}
\]
and $b\pp{d}  := 0$ if $b = 0$.
A sequence of non-negative integers $\left (h_j\right )_{j \geq 0}$ is called
an {\em O-sequence} if $h_0 = 1$ and $h_{j+1} \leq h_j\pp{j}$ for all $j \geq 1$. Macaulay (see, e.g., \cite[Theorem~4.2.10]{BH}) showed that, for a numerical function $h: \Z \to \Z$, the following conditions are equivalent:
\begin{itemize}
\item[(a)] $h$ is the Hilbert function of a standard graded $K$-algebra $R/I$, that is, $\dim_K [R/I]_j = h(j)$ for all integers $j$;
\item[(b)] $h(j) = 0$ if $j < 0$ and $\{h(j)\}_{j \geq 0}$ is an $O$-sequence.
\end{itemize}
\smallskip

Finally, we consider Ferrers hypergraphs. Ferrers graphs are parametrized by partitions and form an important class of bipartite graphs. Their edge ideals admit an explicit minimal free resolution (see \cite{CN1}). They can be specialized to the edge ideals of threshold graphs (see \cite{CN2}). These results have been extended to $d$-partite hypergraphs, $ d\ge 2$, in \cite{NR1}.

\begin{dfn}
A {\em Ferrers hypergraph} is a $d$-partite $d$-uniform hypergraph $F$ on a vertex set $X^{(1)}\sqcup...\sqcup X^{(d)}$  such that there is a linear ordering on each $X^{(j)}$ and whenever
$(i_1,...i_d)\in F$ and $(i_1',...,i_d')$ satisfies $i_j'\leq i_j$ in $X^{(j)}$ for all $j$, one also has $(i_1',...,i_d')\in F$. In other words, $F$ is an order ideal in the componentwise partial ordering on $X^{(1)}\times...\times X^{(d)}$.

The ideal $I(F)$ generated by all the monomials $x_{i_1}^{(1)} \cdots x_{i_d}^{(d)}$, where $(i_1,...i_d)\in F$, is called a
{\em (generalized) Ferrers ideal}.
\end{dfn}

We may assume that the sets $X^{(j)}$ consist of consecutive positive integers $1,2,\ldots,n_j$.

\begin{eg}
  \label{exa:Ferrers-graph}
Monomial ideals generated by variables correspond to 1-uniform Ferrers hypergraphs. Ferrers tableaux describe 2-uniform Ferrers graphs, whereas 3-uniform Ferrers hypergraphs correspond to cubical stackings. For example, consider the stacking
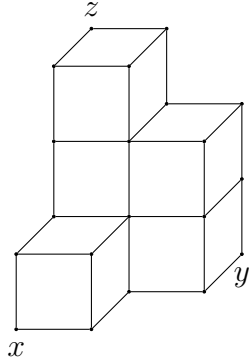
\begin{center}
  \begin{tikzpicture}[scale=1, vertices/.style={draw, fill=black, circle, inner sep=0pt}]
    \node [vertices] (6) at (-1/2,1/2){};
    \node [vertices] (7) at (1/2,-1/2){};
    \node [vertices] (8) at (1/2,1/2){};
    \node [vertices] (10) at (1,2){};
    \node [vertices] (11) at (-1/2,3/2){};
    \node [vertices] (12) at (1/2,3/2){};
    \node [vertices, label={$z$}] (13) at (0,3){};
    \node [vertices] (14) at (1,3){};
    \node [vertices] (15) at (-1/2,5/2){};
    \node [vertices] (16) at (1/2,5/2){};
    \node [vertices, label=below:{$x$}] (17) at (-1,-1){};
    \node [vertices] (18) at (0,-1){};
    \node [vertices] (19) at (-1,0){};
    \node [vertices] (20) at (0,0){};
    \node [vertices, label=below:{$y$}] (21) at (2,0){};
    \node [vertices] (22) at (2,1){};
    \node [vertices] (23) at (3/2,1/2){};
    \node [vertices] (24) at (3/2,-1/2){};
    \node [vertices] (25) at (2,2){};
    \node [vertices] (26) at (3/2,3/2){};
  \foreach \to/\from in {6/8,7/8, 10/12,12/11,11/6,12/8, 13/14,14/16,16/15,15/13, 14/10,15/11,16/12, 17/18,18/20,20/19,19/17,18/7,19/6,20/8, 21/22,22/23,23/24,24/21,23/8,24/7, 25/26,25/22,25/10,26/23,26/12}
  \draw [-] (\to)--(\from);
\end{tikzpicture}
\end{center}
Using variables $x_1, x_2,\ldots$, $y_1, y_2,\ldots$, and $z_1, z_2,\ldots$ to avoid super scripts, the associated Ferrers ideal is
\[
I(F) = (x_1y_1z_1,x_1y_1z_2,x_1y_1z_3,x_1y_2z_1,x_1y_2z_2,x_2y_1z_1).
\]
\end{eg}

Recall that a monomial ideal $I \subset R$ is said to be
\emph{strongly stable} if, for any monomial $u \in S$, the
conditions $u \in I$ and $x_i$ divides $u$ imply that $x_j \cdot
\frac{u}{x_i}$ is in $I$ whenever $j \le i$. A squarefree monomial ideal $I \subset R$ is said to be
\emph{squarefree strongly stable} if, for any squarefree monomial $u \in R$, the
conditions $u \in I$, $x_i$ divides $u$, and $x_j$ does not divide $u$ imply that $x_j \cdot
\frac{u}{x_i}$ is in $I$ whenever $j \le i$.

\begin{rmk}
   \label{rem:polarization} It is well-known how to associate to a given strongly stable ideal a squarefree strongly stable ideal in a ring with enough variables that has the same graded Betti numbers. Indeed, define a map
\begin{equation*}
  \ffi: \{\text{monomials}\} \longrightarrow \{\text{squarefree~monomials}\}
\end{equation*} by
\begin{equation*}
  x_{i_1} x_{i_2} \cdots x_{i_j} \mapsto x_{i_1} x_{i_2 + 1} \cdots x_{i_j + j-1},\quad  \text{where  } 1 \le i_1 \le i_2 \cdots \le i_j.
\end{equation*}
If $I$ is a strongly stable ideal with minimal generators $u_1,...,u_t$, then $J=(\phi(u_1),...,\phi(u_t))$ is a squarefree strongly stable ideals with the same graded Betti numbers as $I$ by Lemmas 1.2 and 2.2 in \cite{AHH1}.
\end{rmk}

According to \cite[Propostion 3.7]{NR1}, every Ferrers hypergraph is isomorphic to a skew squarefree strongly stable hypergraph. However, here we need a different construction.

\begin{rmk}
  \label{rem:sqr-stable-Ferrers}
Using new variables $x_i^{(j)}$, where $i \ge 1$ and $1 \le j \le d$,  consider the map $\psi$ defined by
\begin{equation*}
  x_{i_1}x_{i_2}\cdot \cdot \cdot x_{i_d}\mapsto x_{i_1}^{(1)}x_{i_2 - i_1}^{(2)}\cdot \cdot \cdot x_{i_d - i_{d-1}}^{(d)}\quad  \text{if  } 1 \le i_1 \le i_2 \cdots \le i_j.
\end{equation*}
If $J$ is a squarefree strongly stable ideal whose minimal generators $v_1,\ldots,v_t$ all have degree $d$, then one checks that the ideal generated by $\psi(v_1),\ldots,\psi(v_t)$ is the Ferrers ideal $I(F)$ of a $d$-uniform Ferrers graph. Moreover, the ideals $J$ and $I(F)$ have the same graded Betti numbers by \cite[Theorem 3.13]{NR1} as their minimal free resolutions can be described by using isomorphic cell complexes.
\end{rmk}

We illustrate the passage from a strongly stable ideal to a Ferrers ideal.

\begin{eg}
    \indent Consider the strongly stable ideal $I=(x_1^3,x_1^2x_2,x_1x_2^2,x_1x_2x_3,x_1^2x_3)$. Applying the above map $\ffi$ to each of its generators we get:
\[
I = (x_1^3,x_1^2x_2,x_1x_2^2,x_1x_2x_3,x_1^2x_3) \mapsto(x_1x_2x_3,x_1x_2x_4,x_1x_3x_4,x_1x_3x_5,x_1x_2x_5)
\]
Using the above map $\psi$ we get a Ferrers ideal. However, to avoid superscripts we use variables $y_i = x_i^{(2)}$  and $z_i = x_i^{(3)}$.   We obtain:
\[
(x_1x_2x_3,x_1x_2x_4,x_1x_3x_4,x_1x_3x_5,x_1x_2x_5) \mapsto I(F) =(x_1y_1z_1,x_1y_1z_2,x_1y_2z_1,x_1y_2z_2,x_1y_1z_3).
\]
\end{eg}

\section{{Ideals with $d$-linear resolutions}}
\label{sec:linear-res}

In this section, our goal is to describe the Boij-S\"oderberg decompositions of the Betti tables of ideals with a $d$-linear resolution and of the tables of their quotient rings.

\begin{notation}
    \label{not:linear-diagram}
We use $\pdk$ to denote the pure diagram representing a $d$-linear resolution, that is,  $\pdk = \pi_{\sigma}$, where $\sigma = (d,d+1,\ldots,d+k)$.
\end{notation}

We begin by considering Ferrers ideals. Notice that the graded Betti numbers of an ideal $I$ over $R$ are the same as the ones of the extension ideal $I R[t]$ over $R[t]$, where $t$ is a new variable. Thus, we may drop the reference to the polynomial ring $R$.

\begin{prop}
\label{thm:BS-decomp-lin-ideal}
    Let $F$ be a $d$-uniform Ferrers hypergraph. Then
    the Boij-S\"oderberg decomposition of the associated Ferrers ideal $I(F)$ is
    \begin{equation*}
       \beta(I(F))=\sum_{k\geq 0}\alpha_k(F)\, k!\ \pdk,
    \end{equation*}
    where:
    \[
    \alpha_k(F):=\# \{(i_1,..,i_d)\in F : \sum_j i_j= k+ d\}.
    \]
\end{prop}

\begin{proof}
    Observe that the non-zero entry in homological degree $i$ of $\pdk$ is $\frac{1}{i! (k-i)!}$. Hence, the $i$-th  non-zero entry in the diagram $\sum_{k\geq 0} \alpha_k(F)\, k!\ \pdk$ is
    \[
    \left[\sum_{k\geq 0}\alpha_k(F)\, k!\ \pdk\right]_i=\sum_{k\geq 0}\alpha_k(F)\binom{k}{i}.
    \]
    Our claim follows because
    \begin{equation*}
      \beta_i (I(F)) = \sum_{k\geq 0}\alpha_k(F)\binom{k}{i}
    \end{equation*}
    by Corollary~3.14 in \cite{NR1}.
\end{proof}

\begin{eg}
   \label{exa:Ferrers-BS}
Consider the 3-uniform Ferrers hypergraph
\begin{equation*}
  F = \{(1,1,1), (1,1,2), (1,1,3), (1,2,1, (1,2,2), (2,1,1)\}.
\end{equation*}
Its Ferrers ideal (see Example \ref{exa:Ferrers-graph})
\[
I(F) = (x_1y_1z_1,x_1y_1z_2,x_1y_1z_3,x_1y_2z_1,x_1y_2z_2,x_2y_1z_1).
\]
has Betti table
$$\beta(I(F))=\begin{array}{c|ccc}
\beta_{i,j} & 0 & 1 & 2\\
\hline
3 & 6 & 7 & 2
\end{array}\ .$$
Since $\alpha_2 (F) = 2,\ \alpha_1 (F) = 3$, and $\alpha_0 (F) = 1$, by Theorem \ref{thm:BS-decomp-lin-ideal}, its Boij-S\"oderberg decomposition is
\begin{align*}
\beta(I(F)) & = \alpha_2 (F) \cdot 2! \cdot
\begin{array}{c|ccc}
\beta_{i,j} & 0 & 1 & 2\\
\hline
3 & \frac{1}{2} & 1 & \frac{1}{2}
\end{array}
+
\alpha_1 (F) \cdot 1! \cdot
\begin{array}{c|ccc}
\beta_{i,j} & 0 & 1 & 2\\
\hline
3 & 1 & 1 & \cdot
\end{array} \\
&
\hspace*{1cm} +
\alpha_0 (F) \cdot 0! \cdot
\begin{array}{c|ccccc}
\beta_{i,j} & 0 & 1 & 2\\
\hline
3 & 1 &\cdot &\cdot
\end{array} \\[1ex]
& = 4 \cdot \begin{array}{c|ccc}
\beta_{i,j} & 0 & 1 & 2\\
\hline
3 & \frac{1}{2} & 1 & \frac{1}{2}
\end{array}
+
3 \cdot\begin{array}{c|ccc}
\beta_{i,j} & 0 & 1 & 2\\
\hline
3 & 1 & 1 & \cdot
\end{array}
+
\begin{array}{c|ccccc}
\beta_{i,j} & 0 & 1 & 2\\
\hline
3 & 1 &\cdot &\cdot
\end{array}\ .
\end{align*}
\end{eg}

The following result characterizes the $\alpha$-sequences of Ferrers graphs as defined in Proposition \ref{thm:BS-decomp-lin-ideal}. Recall that an $O$-sequence $h_0, h_1, \ldots$ is a sequence of non-negative integers such that $h_0=1$ and $h_{i+1}\leq h_i^{<i>}$ for all $i \ge 1$ (see Section \ref{sec:prelim}).

\begin{prop}
   \label{prop:charact-alpha-sequence}
Let $(h_0,...,h_s)$ be a sequence of non-negative integers. Then the following conditions are equivalent:
\begin{itemize}
  \item[(a)] The given sequence is the $\alpha$-sequence of a d-uniform Ferrers hypergraph, that is, there is such a graph $F$ such that $\alpha_i (F) = h_i$ whenever $0 \le i \le s$ and $\alpha_i = 0$ if $i > s$.

  \item[(b)] The given sequence is an $O$-sequence with $h_1 \le d$.
\end{itemize}
\end{prop}

\begin{proof}
Denote by $M$ the set of monomials in the polynomial ring $R=K[x_1,...,x_d]$ and consider the map
\begin{align*}
  \ffi: M  \longrightarrow  S := K[x_i^{(j)}|1\leq j\leq d], \quad
  {\bf x}^{\bf a} = x_1^{a_1} x_2^{a_2} \cdots x_d^{a_d} \mapsto  x_{a_1+1}^{(1)}x_{a_2+1}^{(2)}\ldots x_{a_d+1}^{(d)}
\end{align*}

First, we show that (b) implies (a). By Macaulay's theorem, Assumption (b) provides that there is a lexsegment ideal $I$ of $R$ such that its Hilbert function satisfies
\begin{equation*}
  \dim_K [R/I]_j = \begin{cases}
    h_j & \text{ if } 0 \le j \le s \\
    0 & \text{ if  } s < j.
  \end{cases}
\end{equation*}
Denote by $L_j$ the monomials of degree $j$ in $M \setminus I$. Let $J \subset S$ be the ideal that is generated by $\ffi (L_0) \cup \ldots \cup \ffi (L_s)$. Note that all minimal generators of $J$ have degree $d$. We claim that $J$ is a Ferrers ideal. Indeed, if $\ffi ({\bf x}^{\bf a})$ is a minimal generator of $J$, then ${\bf x}^{\bf a} \notin I$. Thus, $\frac{{\bf x}^{\bf a}}{x_i} \notin I$ for each variable $x_i$, so $\ffi (\frac{{\bf x}^{\bf a}}{x_i}) \in J$.

Let $F$ be the $d$-uniform Ferrers graph such that $J = I(F)$. Then $\alpha_i (F) = h_i$ follows from the construction of $F$.

Second, we assume (a) and show (b). Let $L \subset R$ be the set of monomials consisting of the preimages under $\ffi$ of the minimal generators of the Ferrers ideal $I(F)$. It consists of monomials whose degree is at most $s$. Let $L_j \subset L$ be the subset of monomials having degree $j$. Then the cardinality of $L_j$ is $\alpha_j (F)$ by  construction. Moreover, observe that $L$ is an order ideal of $M$ with respect to the partial order given by divisibility because $F$ is a Ferrers hypergraph.

Let $I \subset R$ be the ideal that is generated by all the monomials that are not in $L$. Then $I$ is an artinian ideal whose inverse system is the order ideal $L$. Thus, we get
\begin{equation*}
  \dim_K [R/I]_j = \# L_j = \alpha_j (F) = h_j.
\end{equation*}
Hence Macaulay's characterization of Hilbert functions implies that $(h_0,...,h_s)$ is an $O$-sequence.
\end{proof}

We are ready for the first main result of this section.   We use Notation \ref{not:linear-diagram} in our characterization of the Betti numbers of ideals with a $d$-linear resolution.

\begin{thm}
  \label{thm:charact-lin-resolution}
Let $R = K[x_1.\ldots,x_n]$ and consider the diagram
\begin{equation*}
  \be = \sum_{k = 0}^v \alpha_k\, k!\ \pdk,
\end{equation*}
where $\al_0,\ldots,\al_v$ are rational numbers and $v \le n$. Then the following conditions are equivalent:
\begin{itemize}
        \item[(a)] $\beta$ is the Betti table of an ideal of $R$ with a $d$-linear resolution.

\item[(b)] $\beta$ is the Betti table of a strongly stable ideal $I$ whose minimal generators have degree $d$.

\item[(c)] $\beta$ is the Betti table of the ideal to a $d$-uniform Ferrers hypergraph $F$ with
\begin{equation*}
\alpha_i (F) = \begin{cases}
  \alpha_i & \text{ if } 0 \le i \le v \\
  0 & \text{ if  } v < i.
\end{cases}
\end{equation*}

\item[(d)] $(\alpha_0,\alpha_{1},...,\alpha_{v})$ is an $O$-sequence with $\alpha_{1}\leq d$.
\end{itemize}
\end{thm}

\begin{proof}
(a) $\Rightarrow$ (b): Let $I$ be an ideal with Betti table $\be$. Then the generic initial ideal of $I$ with respect to the reverse lexicographic order has the same graded Betti numbers as $I$ (see \cite[Corollary 4.3.18]{HH} and \cite[Corollary 6.1.5]{HH}). Furthermore, if $K$ has characteristic zero, then $\gin I$  is strongly stable, and we are done. If the characteristic of $K$ is positive, then $\gin I$ is at least a stable monomial ideal. This follows, for example, by \cite[Theorem 2.5]{NR2}. Consider now $\gin I$ as an ideal in a polynomial ring whose base field, $L$, has characteristic zero. Since the minimal free resolution of $\gin I$, as described by Eliahou and Kervaire, does not depend on the characteristic, the Betti numbers of $\gin I$ remain the same when considered over $L$. Passing now to the generic initial ideal with respect to the reverse lexicographic order gives the desired strongly stable ideal.

(b) $\Rightarrow$ (c): Remarks \ref{rem:polarization} and \ref{rem:sqr-stable-Ferrers} provide to each strongly stable ideal whose generators have degree $d$ a Ferrers ideal to a $d$-uniform hypergraph with the same graded Betti numbers.

(c) $\Rightarrow$ (a):  This is true by Proposition \ref{thm:BS-decomp-lin-ideal}.

Conditions (c) and (d) are equivalent by Proposition \ref{prop:charact-alpha-sequence}.
\end{proof}

\begin{rmk}
A related, but compared to  Theorem \ref{thm:charact-lin-resolution} different characterization of the Betti tables of ideals with a $d$-linear resolution has been established by Herzog, Sharifan, and Varbaro in \cite[Theorem 3.2]{HSV}. It uses combinatorial information on the generators of a strongly stable monomial ideal.
\end{rmk}

Since the Betti numbers of the quotient ring $R/I$ are determined by the Betti numbers of the ideal $I$, one might expect the decompositions of the Betti tables to be similar or, at least, related. However, in general the precise relationship is not known. We solve this problem if the ideal $I$ has a $d$-linear resolution. By the previous result, we may assume that $I$ is a Ferrers ideal. In this case we show that the decomposition of the quotient ring can be found by counting the same set that defined the numbers $\alpha_k (F)$, just in a different fashion.

In order to state the result, we need some notation.

\begin{notation}
  \label{not:quotient-linear}
We use $\qdk$ to denote the pure diagram  $\pi_{\underline{d}}$, where the degree sequence is $\underline{d} = (0,d,d+1,\ldots,d+k_S)$.
\end{notation}

In the following result we exclude the case $d=1$ in which the Boij-S\"oderberg decomposition is trivial. It has only one summand.

\begin{thm}
   \label{thm:BS-quotient-linear}
Let $F$ be a $d$-uniform Ferrers hypergraph on the vertex set $X^{(1)} \sqcup \ldots \sqcup X^{(d)}$, where $d \ge 2$.   Then the Boij-S\"oderberg decomposition of the quotient ring $R/I(F)$ is
\begin{equation*}
  \beta(R/I(F))=\sum_{j=1}^d \sum_{S \in F_j} n_S \cdot k_S! \cdot \qdk,
\end{equation*}
where $F_j$ is the Ferrers hypergraph
\begin{align*}
  F_j  & :=\{(i_1,\ldots,\widehat{i_j},\ldots,i_d) : \text{There is some $i_j \in X^{(j)}$ such that }  (i_1,\ldots,i_j,\ldots,i_d) \in F\},  \\
n_S & := \max \{i_j \in X^{(j)} : (i_1,\ldots,i_j,\ldots,i_d) \in F\} \; \text{ if } S = (i_1,\ldots,\widehat{i_j},\ldots,i_d) \in F_j, \; \text{and} \\
k_S & := n_S - d + \sum_{p=1,p\neq j}^d i_p.
\end{align*}
\end{thm}

\begin{proof}
Note that the non-zero entries in $\qdk$ are
\begin{equation}
    \label{eq:pure-Betti}
\be_i (\qdk) = \begin{cases}
   \displaystyle{\frac{(d-1)!}{(d+k_S)!}} & \text{ if  } i = 0 \\
   \displaystyle{\frac{1}{(i-1)! \cdot (k_S-i+1)! \cdot (d+i-1)}} & \text{ if  } 1 \le i \le k_S +1.
\end{cases}
\end{equation}
Thus, the entry  of $\sum_{j=1}^d \sum_{S \in F_j} n_S \cdot k_S! \cdot \qdk$ in homological degree $i$ is
\begin{equation*}
  \be_i \left (\sum_{j=1}^d \sum_{S \in F_j} n_S \cdot k_S! \cdot \qdk \right ) = \frac{1}{d+i-1} \sum_{j=1}^d \sum_{S \in F_j} n_S \cdot \binom{k_S}{i-1}.
\end{equation*}

Consider first the Betti numbers with positive index, i.e., assume that $i \ge 1$. Then \cite[Corollary~3.14]{NR1} gives that
\begin{equation*}
  \beta_i (R/I(F)) = \sum_{(i_1,\ldots,i_d) \in F} \binom{\sum_p i_p - d}{i-1}.
\end{equation*}
It follows that we have to show the identity
\begin{equation}
    \label{eq:betti-ident-pos}
(d + i -1) \sum_{(i_1,\ldots,i_d) \in F} \binom{\sum_p i_p - d}{i-1}  = \sum_{j=1}^d \sum_{S \in F_j} n_S \cdot \binom{k_S}{i-1}.
\end{equation}

To this end denote by $N$ the number of possibilities for choosing pairs $(X, y)$, where $y \in X$ and $X$ is a subset of $X^{(1)} \sqcup \ldots \sqcup X^{(d)}$ with cardinality $d+i-1$ and, for each $p$, maxima $m_p =  \max (X \cap X^{(p)})$ in $X^{(p)}$ such that  $(m_1,\ldots,m_d) \in F$. We establish Identity \eqref{eq:betti-ident-pos} by determining $N$ in two different ways.
\smallskip

{\em Approach 1.}: We classify the possible subsets $X$ according to their maxima in each set $X^{(p)}$.

Fix  $(m_1,\ldots,m_d) \in F$. To extend $\{m_1,\ldots,m_d\}$ to a subset $X$ with maxima $m_1,\ldots,m_d$, we can choose $i-1$ numbers among any of the first $m_p - 1$ elements in each $X^{(p)}$. There are $\binom{\sum_p m_p - d}{i-1}$ such choices. Taking into account the number of choices for $y \in X$, we conclude that
\begin{equation}
  \label{eq:count1}
N = (d + i -1) \sum_{(m_1,\ldots,m_d) \in F} \binom{\sum_p m_p - d}{i-1}.
\end{equation}
\smallskip

{\em Approach 2.}: This time we classify the possibilities for choosing  $(X, y)$  according to the number $j$ such that $y \in X^{(j)}$ and the maxima of $X$ in all $X^{(p)}$, except $X^{(j)}$.

Fix $j \in \{1,\ldots,d\}$ and $S = (m_1,\ldots,\widehat{m_j},\ldots,m_d) \in F_j$.  We want to pick $y$ in $S^{(j)}$ and extend $\{m_1,\ldots,\widehat{m_j},\ldots,m_d, y\}$ to a subset $X$ with $d+i-1$ elements and maxima vector $(m_1,\ldots, \max (X \cap X^{(j)}),\ldots,m_d)$ in $F$. In order to ensure the latter condition, all elements in $X \cap X^{(j)}$ have to be among the first $n_S$ elements of $X^{(j)}$ by definition of $n_S$ and using the defining property of a Ferrers hypergraph. Thus, there are $n_S$ choices for $y$ in $S^{(j)}$. The other $i-1$ numbers in $X$ can be chosen among any of the first $m_p - 1$ elements in each $X^{(p)}$ if $p \neq j$ and among the first $n_S$ elements in $X^{(j)}$, except $y$. We conclude that
\begin{equation}
  \label{eq:count2}
N = \sum_{j=1}^d \sum_{S \in F_j} n_S \cdot \binom{n_S - d + \sum_{p=1,p\neq j}^d i_p}{i-1} =  \sum_{j=1}^d \sum_{S \in F_j} n_S \cdot \binom{k_S}{i-1}
\end{equation}
\smallskip

Comparing Equations \eqref{eq:count1} and \eqref{eq:count2}, we obtain the desired Identity \eqref{eq:betti-ident-pos}.

It remains to consider the $0$-th Betti number. However, the alternating sum of the total Betti numbers in a minimal free resolution is zero. Hence our claim for $i =0$ follows from our results for $i \ge 1$.
\end{proof}

\begin{rmk}
\indent Theorem \ref{thm:BS-quotient-linear} extends the conclusions of group 10.2 (E.Celikbas, D. Linsay, S. Sanyal, S. Sturgeon, K. Yu) at the MSRI summer workshop in commutative algebra 2011. In their report they show the conclusion in the case $d = 2$. \end{rmk}

We illustrate the last result in case $d = 3$.

\begin{eg}
  \label{ex:BS-Ferrers}
\indent Consider again the ideal $I(F) = ( x_1y_1z_1,x_1y_1z_2,x_1y_1z_3,x_1y_2z_1,x_1y_2z_2,x_2y_1z_1)$, corresponding to the cubical stacking \\
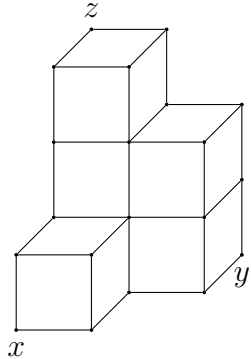
\begin{center}
\begin{tikzpicture}[scale=1, vertices/.style={draw, fill=black, circle, inner sep=0pt}]
    \node [vertices] (6) at (-1/2,1/2){};
    \node [vertices] (7) at (1/2,-1/2){};
    \node [vertices] (8) at (1/2,1/2){};
    \node [vertices] (10) at (1,2){};
    \node [vertices] (11) at (-1/2,3/2){};
    \node [vertices] (12) at (1/2,3/2){};
    \node [vertices, label={$z$}] (13) at (0,3){};
    \node [vertices] (14) at (1,3){};
    \node [vertices] (15) at (-1/2,5/2){};
    \node [vertices] (16) at (1/2,5/2){};
    \node [vertices, label=below:{$x$}] (17) at (-1,-1){};
    \node [vertices] (18) at (0,-1){};
    \node [vertices] (19) at (-1,0){};
    \node [vertices] (20) at (0,0){};
    \node [vertices, label=below:{$y$}] (21) at (2,0){};
    \node [vertices] (22) at (2,1){};
    \node [vertices] (23) at (3/2,1/2){};
    \node [vertices] (24) at (3/2,-1/2){};
    \node [vertices] (25) at (2,2){};
    \node [vertices] (26) at (3/2,3/2){};
  \foreach \to/\from in {6/8,7/8, 10/12,12/11,11/6,12/8, 13/14,14/16,16/15,15/13, 14/10,15/11,16/12, 17/18,18/20,20/19,19/17,18/7,19/6,20/8, 21/22,22/23,23/24,24/21,23/8,24/7, 25/26,25/22,25/10,26/23,26/12}
  \draw [-] (\to)--(\from);
\end{tikzpicture}
\end{center}
The Betti table of $R/I(F)$  is
\[
\beta(R/I(F))=\begin{array}{c|cccc}
\beta_{i,j} & 0 & 1 & 2 & 3\\
\hline
0 & 1 & \cdot & \cdot & \cdot\\
1 & \cdot & \cdot & \cdot & \cdot\\
2  & \cdot & 6 & 7 & 2
\end{array} .
\]

Abusing notation by, for example, identifying $(i,j,k) \in F$ with the monomial $x_i y_j z_k$, we get the following data for the Ferrers graphs $F_1, F_2$, and $F_3$:
\[
F_1: \quad \begin{array}{c|ccccc}
S & y_1z_1 & y_1z_2 & y_1z_3 & y_2z_1 & y_2z_2 \\
\hline
n_S & 2 & 1& 1 & 1 & 1 \\
k_S & 1 & 1 & 2 & 1 & 2
\end{array},
\]
\[
F_2: \quad \begin{array}{c|cccc}
S & x_1z_1 & x_1z_2 & x_1z_3 & x_2z_1 \\
\hline
n_S & 2 & 2& 1 & 1  \\
k_S & 1 & 2 & 2 & 1
\end{array},
\]
and
\[
F_3: \quad \begin{array}{c|ccc}
S & x_1y_1 & x_1y_2 & x_2y_1\\
\hline
n_S & 3 & 2 & 1\\
k_S & 2 & 2 & 1
\end{array} .
\]
Since $20=(n_{y_1z_3}+n_{y_2z_2}+n_{x_1z_2}+n_{x_1z_3}+n_{x_1y_1}+n_{x_1y_2}) \cdot 2!$ and $8=(n_{y_1z_1}+n_{y_1z_2}+n_{y_2z_1}+n_{x_1z_1}+n_{x_2z_1}+n_{x_2y_1}) \cdot 1!$, Theorem \ref{thm:BS-quotient-linear} yields the Boij-S\"oderberg decomposition
\[
\beta(R/I(F))=20 \cdot \begin{array}{c|cccc}
\beta_{i,j} & 0 & 1 & 2 & 3\\
\hline
0 & \frac{1}{60} & \cdot & \cdot & \cdot\\
1 & \cdot & \cdot & \cdot & \cdot\\
2  & \cdot & \frac{1}{6} & \frac{1}{4} & \frac{1}{10}
\end{array} +
8 \cdot \begin{array}{c|ccc}
\beta_{i,j} & 0 & 1 & 2  \\
\hline
0 & \frac{1}{12} & \cdot & \cdot \\
1 & \cdot & \cdot & \cdot \\
2 & \cdot & \frac{1}{3} & \frac{1}{4}
\end{array}.
\]

Notice that the Boij-S\"oderberg decomposition of the Betti table of $I(F)$ has three summands (see Example \ref{exa:Ferrers-BS}), whereas the one of $R/I(F)$ has only two summands.
\end{eg}

Theorem \ref{thm:BS-quotient-linear} provides a curious identity for each Ferrers hypergraph.

\begin{cor}
   \label{cor:curious-Ferrers-identity}
Let $F$ be a $d$-uniform Ferrers hypergraph and adopt the notation of
Theorem \ref{thm:BS-quotient-linear}. Then
\begin{equation*}
  \label{eq:identity-gen}
d =    \sum_{j=1}^d \sum_{S \in F_j}  \frac{n_S}{\binom{d+ k_S}{d}}.
\end{equation*}
\end{cor}

\begin{proof} Considering the $0$-th Betti numbers in Theorem \ref{thm:BS-quotient-linear} and using Equation \eqref{eq:pure-Betti} we get the identity
\begin{equation*}
  \label{eq:identity-gen}
1 =    \sum_{j=1}^d \sum_{S \in F_j} n_S \cdot k_S! \cdot \frac{(d-1)!}{(d+ k_S)!}.
\end{equation*}
Our claim follows.
\end{proof}

Our argument for Corollary \ref{cor:curious-Ferrers-identity} is rather indirect. There also is a more direct argument to establish this identity using induction on the number of vertices in the Ferrers hypergraph.


\section{Quasi-Gorenstein modules}
\label{sec:quasi-Gor}

Quasi-Gorenstein modules were introduced in \cite{mod-liai} as the graded perfect $R$-modules that  are isomorphic to a degree shift of their canonical module. A cyclic module $R/I$ is quasi-Gorenstein if and only if $I$ is a Gorenstein ideal. In module liaison theory quasi-Gorenstein modules assume the role Gorenstein ideals play in Gorenstein liaison theory. These modules have a self-dual minimal free resolution. The goal of this section is to show that this self-duality is reflected in the Boij-S\"oderberg decomposition of the Betti table.

Let $M$ be a finitely generated graded module over $R = K[x_1,\ldots,x_n]$. We denote its $R$-dual $\Hom_R (M, R)$ by $M^*$. It also is a graded module. We call $c = \dim R - \dim M = n - \dim M$ the codimension of $M$. The canonical module of $M$ is $K_M = \Ext^c_R (M, R)(-n)$. If $M$ is Cohen-Macaulay and there is an integer $t$ such that $M \cong K_M (t)$, then $M$ is said to be a {\em quasi-Gorenstein module}.

Let now $M$ be a Cohen-Macaulay $R$-module of codimension $c$ with minimal free resolution
\begin{equation*}
  0 \longrightarrow F_c \stackrel{\ffi_c}{\longrightarrow} F_{c-1} \longrightarrow \cdots \stackrel{\ffi_1}{\longrightarrow} F_0 \longrightarrow M \longrightarrow 0.
\end{equation*}
Dualizing with respect to $R$ we get the minimal free resolution
\begin{equation}
   \label{eq:dual-res}
  0 \longrightarrow F_1^* \stackrel{\ffi_1^*}{\longrightarrow} F_{2}^* \longrightarrow \cdots \stackrel{\ffi_c^*}{\longrightarrow} F_c^* \longrightarrow \Ext^c_R (M, R) \longrightarrow 0
\end{equation}
because $\Ext^i_R (M, R) = 0$ whenever $i \neq c$ as $M$ is Cohen-Macaulay.
Hence, if $M$ is a quasi-Gorenstein module, then the two free resolutions are isomorphic as exact sequences, up to a degree shift. It follows that the free modules $F_i$ and $F_{c-i}^*$ are isomorphic, up to a degree shift that is independent of $i$. The resulting self-duality of the free resolution means in particular that, for all integers $i$ and $j$,
\begin{equation}
  \label{eq:self-duality}
\be_{i, j} (M) = \be_{c-i, m - j} (\Ext^c_R (M, R)),
\end{equation}
where $m = \reg M + c +  a(M)$. Here $a(M)$ denotes the least degree of a minimal generator of $M$ and $\reg M = -c + \max \{j : \be_{c, j} (M) \neq 0\}$ its Castelnuovo-Mumford regularity. In order to capture this self-duality of the free resolution of $M$ in the Boij-S\"oderberg decomposition, we introduce.

\begin{dfn}
  \label{def:dual-pure}
Consider the pure diagram $\pi_{\sigma}$ to the degree sequence $\sigma = (d_0,d_1,\ldots,d_c)$. Then its {\em dual pure diagram} is the pure diagram $\pi_{\sigma^*}$, where $\sigma^* := (-d_c,\ldots,-d_1,-d_0)$.

Moreover, for any integer $m$, we denote by $\pi_{m + \sigma}$ the pure diagram to the degree sequence $m + \sigma := (m+d_0,m+d_1,\ldots,m+d_c)$.
\end{dfn}

We note that the new pure Betti diagrams have the following properties.

\begin{lem}
  \label{lem:Betti-number-dual-pure}
\begin{itemize}
  \item[(a)] For each $i$, \quad $\be_i (\pi_{\sigma}) = \be_{c-i} (\pi_{\sigma^*})$.

  \item[(b)] For each integer $m$ and each $i$, \quad $\be_i (\pi_{\sigma}) = \be_i (\pi_{m + \sigma})$.
\end{itemize}
\end{lem}

\begin{proof}
  Both claims follow directly from the definition of the pure Betti diagrams.
\end{proof}

We will refer to any Betti diagram of the form $\pi_{\sigma} + \pi_{m + \sigma^*}$ as a self-dual Betti diagram. This is justified by comparing Equation \eqref{eq:self-duality} with the following observation.

\begin{cor}
  \label{cor:self-dual-table}
For each integers $i$ and $j$, the entries of the diagram $\pi_{\sigma} + \pi_{m + \sigma^*}$ satisfy
\begin{equation*}
  \be_{i, i+j} = \be_{c-i, m-j}
\end{equation*}
\end{cor}

\begin{proof}
  This is a consequence of Lemma \ref{lem:Betti-number-dual-pure}.
\end{proof}

We are ready for the main result if this section.

\begin{thm}
  \label{thm:BS-self-dual}
Let $M$ be a quasi-Gorenstein module of codimension $c$.  Set $m = \reg M + c +  a(M)$. Then the Boij-S\"oderberg decomposition of $M$ is an integer linear combination of self-dual Betti diagrams of the form $\pi_{\sigma} + \pi_{m + \sigma^*}$ and, in case the number of Boij-S\"oderberg summands of $M$ is odd, a pure diagram $\pi_{\sigma}$ such that $\pi_{\sigma} = \pi_{m + \sigma^*}$.
\end{thm}

\begin{proof}
Consider the Boij-S\"oderberg decomposition of the Betti table of $M$
\begin{equation*}
  \label{eq:BS-M}
\be (M) = \sum_{i=1}^t a_i \pi_{\sigma_i},
\end{equation*}
where $\pi_{\sigma_1} < \pi_{\sigma_2} < \cdots < \pi_{\sigma_t}$.

Observe that, for any pure diagrams $\pi_{\sigma}$ and $\pi_{\tau}$, the relation $\pi_{\sigma} < \pi_{\tau}$ implies $\pi_{\tau^*} < \pi_{\sigma^*}$.  It follows (see Sequence \eqref{eq:dual-res}) that the Betti table of $\Ext^c_R (M, R)$ has the Boij-S\"oderberg decomposition
\begin{equation*}
  \label{eq:BS-M}
\be (\Ext^c_R (M, R)) = \sum_{i=1}^t a_i \pi_{\sigma_i^*},
\end{equation*}
where $\pi_{\sigma_1^*} > \pi_{\sigma_2^*} > \cdots > \pi_{\sigma_t^*}$.

By assumption, $M$ is quasi-Gorenstein, and thus $\Ext^c_R (M, R)(-m) \cong M$ as graded $R$-modules. Hence, comparing the above decompositions and using the uniqueness of the Boij-S\"oderberg decomposition, we conclude that, for all $i$,
\[
\pi_{\sigma_i} = \pi_{m + \sigma_{t+1-i}^*}.
\]
Our claim follows.
\end{proof}

If $M$ is a Cohen-Macaulay module, then $M \oplus K_M(j)$ is a quasi-Gorenstein module for each integer $j$ (see \cite[Remark 2.5(iii)]{mod-liai}). Its Boij-S\"oderberg decomposition is determined by the one of $M$. Notice that the number of summands in the decomposition of $M \oplus K_M(j)$ is always even. This is not true for arbitrary quasi-Gorenstein modules. In the next section, we will exhibit explicit examples in the case of cyclic quasi-Gorenstein modules. Such a cyclic module is isomorphic to a Gorenstein ring, up to a degree shift. Here we give an example arising in the birational geometry of surfaces.

\begin{eg}
  \label{ex:quasi-Gor}
Let $S$ be a regular surface of general type such that the  canonical map is a birational morphisms onto its image  $Y \subset \PP^4$. If the geometric genus of $S$ is five and $K_S^2 = 11$, then the canonical ring $M = \oplus_{m \ge 0} H^0 (S, \cO_S (m K_S))$ is a quasi-Gorenstein module over the coordinate ring $R$ of $\PP^4$ with minimal free resolution of the form (see \cite[Theorem 1.5]{B})
\begin{equation*}
  0 \to R(-6) \oplus R^2 (-4) \to R^6 (-3) \to R \oplus R^2(-2) \to M \to 0.
\end{equation*}
The Boij-S\"oderberg decomposition of its Betti table is
\begin{align*}
  \be (M) & =  \begin{array}{c|ccc}
\beta_{i,j} & 0 & 1 & 2 \\
\hline
0 & 1 & \cdot & \cdot \\
1 & \cdot & \cdot & \cdot \\
2 & 2  &  6 &  2 \\
3 & \cdot & \cdot & \cdot \\
4 & \cdot & \cdot & 1
\end{array} \\
& =  8 \cdot \begin{array}{c|ccc}
\beta_{i,j} & 0 & 1 & 2 \\
\hline
0 & \frac{1}{12} & \cdot & \cdot \\
1 & \cdot & \cdot & \cdot \\
2 & \cdot  &  \frac{1}{3} &  \frac{1}{4} \\
3 & \cdot & \cdot & \cdot \\
4 & \cdot & \cdot & \cdot
\end{array} \;
+ 6 \cdot \begin{array}{c|ccc}
\beta_{i,j} & 0 & 1 & 2 \\
\hline
0 & \frac{1}{18} & \cdot & \cdot \\
1 & \cdot & \cdot & \cdot \\
2 & \cdot  &  \frac{1}{9} &  \cdot  \\
3 & \cdot & \cdot & \cdot \\
4 & \cdot & \cdot & \frac{1}{18}
\end{array} \;
+ 8 \cdot \begin{array}{c|ccc}
\beta_{i,j} & 0 & 1 & 2 \\
\hline
0 & \cdot & \cdot & \cdot \\
1 & \cdot & \cdot & \cdot \\
2 & \frac{1}{4}  &  \frac{1}{3} &  \cdot \\
3 & \cdot & \cdot & \cdot \\
4 & \cdot & \cdot & \frac{1}{12}
\end{array} .
\end{align*}
\end{eg}
Observe that the third summand is $\pi_{(2, 3, 6)} = \pi_{6 + (0,3,4)^*}$, where $\pi_{(0,3,4)}$ is the first summand, and that the second summand satisfies
$\pi_{(0,3,6)} = \pi_{6 + (0,3,6)^*}$, as predicted by Theorem~\ref{thm:BS-self-dual}.


\section{{Gorenstein rings}}
\label{sec:Gor-rings}

In this section we consider Gorenstein rings whose minimal free resolutions has at most two linear strands when one ignores the first and the last homological degree. Gorenstein rings with such a resolution occur naturally. In fact, any such resolution is the minimal free resolution of the Stanley-Reisner ring associated to the boundary complex of a simplicial polytope according to \cite{MN}.

The following result describes the Betti tables whose Boij-S\"oderberg decomposition we derive in this section.

\begin{lem}
  \label{lem:Gor-res}
Let $s, t, c$ be positive integers such that $s \ge 2t$ and $c \le n$. Then there is a homogeneous Gorenstein ideal $I \subset R$ of codimension $c$ such that the graded minimal resolution of $R/I$ has the shape
\begin{eqnarray*}
\lefteqn{ 0 \longrightarrow R(-s-c) \longrightarrow
\begin{matrix}
  R^{a_{c-1}} (-t-c+1) \\
  \oplus \\
  R^{a_1} (-s+t-c+1)
\end{matrix}
\longrightarrow \cdots } \\[1ex]
& & \hspace*{1.9cm}  \longrightarrow
\begin{matrix}
  R^{a_{2}} (-t-2) \\
  \oplus \\
  R^{a_{c-2}} (-s+t-2)
\end{matrix}
\longrightarrow
\begin{matrix}
  R^{a_{1}} (-t-1) \\
  \oplus \\
  R^{a_{c-1}} (-s+t-1)
\end{matrix}
\longrightarrow R \longrightarrow R/I \longrightarrow 0,
\end{eqnarray*}
where, for $i = 1,\ldots,c-1$,
\begin{equation*}
  a_i = \binom{c+t-1}{i+t}\binom{t-1+i}{t}.
\end{equation*}
\end{lem}

\begin{proof}
This follows by \cite[Theorem 8.13]{MN}.
\end{proof}

Ideals with this minimal free resolution arise in various ways.
Recall that the Hilbert series of any graded $K$-algebra $R/I$ can be uniquely written as
\begin{equation*}
  H_{R/I} (t) := \sum_{j \ge 0} \dim_K [R/I]_j = \frac{h(t)}{(1-t)^d},
\end{equation*}
where $h (1) \neq 0$, $h(t) = h_0 + h_1 t + \cdots h_r t^r \in \Z[t]$, and $d = \dim R/I$. The coefficient vector $(h_0,h_1,\ldots,h_r)$ is called the {\em $h$-vector} of $R/I$.
\begin{rmk}
  \label{rem:hilb-Gor}
Assume $R/I$ is a Gorenstein ring with a free resolution as in Lemma \ref{lem:Gor-res}. Then its $h$-vector $h=(h_0,...,h_s)$ is given by
\begin{equation*}
h_i = \begin{cases}
  \binom{c-1+i}{c-1} & \text{ if }   0 \leq i \leq t;\\[1ex]
  \binom{c-1+t}{c-1} & \text{ if }   t\leq i\leq s-t;\\[1ex]
  \binom{s-i+c-1}{c-1} & \text{ if }  s-t \leq i \le s.
\end{cases}
\end{equation*}
\end{rmk}
\medskip

Conversely, Gorenstein rings with this $h$-vector are often forced to have a free resolution as described in Lemma \ref{lem:Gor-res}. To this recall that a graded Gorenstein algebra $R/I$ of dimension $d$ has the {\em weak Lefschetz property} if there are linear forms $\ell, \ell_1,\ldots,\ell_d$ such that $A = R/(I, \ell_1,\ldots,\ell_d)$ has dimension zero and, for each $j$, the multiplication map
\begin{equation*}
  \times \ell: [A]_{j-1} \longrightarrow [A]_j, \quad a \mapsto \ell a,
\end{equation*}
has maximal rank, that is, it is injective or surjective.

\begin{thm}[\mbox{\cite[Corollary 8.14]{MN}}]
Let $c,s,t$ be positive integers, where either $s = 2t$ or $s \ge 2t+2$. Let $R/I$ be a Gorenstein algebra of dimension $d = n-c$ with an $h$-vector as in Remark \ref{rem:hilb-Gor}. If $R/I$ has the weak Lefschetz property, then $R/I$ has a minimal free resolution as in Lemma \ref{lem:Gor-res}.
\end{thm}

Furthermore, it follows by \cite[Theorem 9.6]{MN} that each of the resolutions described in Lemma \ref{lem:Gor-res} occurs as the minimal free resolution of the Stanley-Reisner ring associated to the boundary complex of a simplicial polytope. Our main result in this section describes the Boij-S\"oderberg decomposition of the corresponding Betti table.

\begin{thm}
   \label{thm:Gor-ideals}
Let $R/I$ be a Gorenstein ring with a free resolution as in Lemma \ref{lem:Gor-res}.  Then the Boij-S\"oderberg decomposition of $R/I$ is
\begin{equation}
   \label{eq:BS-Gor-decomp}
 \beta(R/I)= a \cdot [\pi_{\sigma_1} + \pi_{\sigma_c}]  + b \cdot \sum_{j=2}^{c-1} \pi_{\sigma_j}, \end{equation}
where
\begin{align*}
  a = & (s+1-t) \frac{(t+c-1)!}{t!}, \\
  b = & (s+1-2t) \frac{(t+c-1)!}{t!},
\end{align*}
and
\begin{equation*}
  \sigma_j = (0,d_{j, 1},\ldots,d_{j, c-1},s+c)
\end{equation*}
with
\[
d_{j, k} = \begin{cases}
  t+k & \text{if } 1 \le k \le c-j \\
  s-t+k & \text{if } c-j + 1 \le k \le c-1
\end{cases}.
\]
\end{thm}

As preparation for its proof, we derive the following identity.

\begin{lem}
  \label{lem:Gor-identity}
If $a, b, m$ are positive integers such that $m \leq a$, then
\begin{equation*}
  \sum_{j=1}^m \frac{\binom{a}{j}}{\binom{a+b}{j}} = - \frac{(a+b+1-m) \binom{a}{m+1}}{(b+1) \binom{a+b}{m+1}} + \frac{a}{b+1}.
\end{equation*}
\end{lem}

\begin{proof}
Define
\[
\mu(j)=-\frac{(a+b-j+1)\binom{a}{j}}{(b+1)\binom{a+b}{j}}.
\]
Then one checks that
\[
\frac{\binom{a}{j}}{\binom{a+b}{j}}=\mu(j+1)-\mu(j).
\]
Hence, we get the telescope sum
\begin{align*}
  \sum_{j=1}^m \frac{\binom{a}{j}}{\binom{a+b}{j}} & =  \sum_{j=1}^m\mu(j+1)-\mu(j) \\
& = \mu (m+1) - \mu (1) \\
& =  - \frac{(a+b+1-m) \binom{a}{m+1}}{(b+1) \binom{a+b}{m+1}} + \frac{a}{b+1},
\end{align*}
as claimed.
\end{proof}

\begin{proof}[Proof of Theorem \ref{thm:Gor-ideals}]
Let $1 \le i \le c-1$, and consider the graded Betti number $\be_{i, t+i} = a_i$. According to Formula \eqref{eq:BS-Gor-decomp}, we claim that precisely the pure Betti tables $\pi_{\sigma_1},\ldots,\pi_{\sigma_{c-i}}$ contribute to this Betti number. For $j=1,\ldots,c-i$, the degree sequence $\sigma_j$ is
\[
\sigma_j=(0,t+1,\ldots,t+i,\ldots,t+c-j,s-t+c-j,\ldots,s-t+c-1,c+s).
\]
Hence, we get for the non-zero entry in homological degree $i$ of $\pi_{\sigma_j}$
\begin{equation*}
  \be_i (\pi_{\sigma_j}) = \frac{(s-2t+c-i-j)!}{(t+i) (s+c-t-i) (i-1)! (c-j-i)! (s-2t+c-1-i)!}.
\end{equation*}
Using that $a = \frac{s+1-t}{s+1-2t} \cdot b$, our claim for $\be_{i, t+i}$ is equivalent to
\begin{equation*}
  \frac{a_i}{b} = \frac{s+1-t}{s+1-2t} \cdot \be_i (\pi_{\sigma_1}) + \sum_{j=2}^{c-i} \be_i (\pi_{\sigma_j}).
\end{equation*}
Simplifying $\frac{a_i}{b}$, this means that we have to show
\begin{eqnarray*}
\lefteqn{ \frac{1}{(t+i) (s+1-2t) (i-1)! (c-j-i)!} =  } \\
& & \hspace*{2.5cm} \frac{s+1-t}{s+1-2t} \cdot \frac{1}{(t+i) (s+c-t-i) (i-1)! (c-1-i)!} \\
& & \hspace*{2.5cm} + \sum_{j=2}^{c-i} \frac{(s-2t+c-i-j)!}{(t+i) (s+c-t-i) (i-1)! (c-j-i)! (s-2t+c-1-i)!}.
\end{eqnarray*}
The latter is equivalent to
\begin{align*}
\frac{s+c-i-t}{s+1-2t}
 & = \frac{s+1-t}{s+1-2t}  + \sum_{j=2}^{c-i} \frac{(s-2t+c-i-j)! (c-1-i)!}{ (s-2t+c-1-i)! (c-j-i)! } \\
& = \frac{s+1-t}{s+1-2t}  + \sum_{j=2}^{c-i} \frac{\binom{c-1-i}{j-1}}{\binom{s-2t+c-1-i}{j-1}} \\
& = \frac{s+1-t}{s+1-2t}  + \frac{c-1-i}{s+1-2t},
\end{align*}
which is certainly true and
where we used Lemma \ref{lem:Gor-identity} with $a = c-1-i$ and $b = s-2t$ to establish  the last equality.

Using the symmetry on both sides of Identity \eqref{eq:BS-Gor-decomp}, it only remains to check the claim for the Betti number $\be_0 (R/I) = 1$. To this end note that, for each $j=1,\ldots,c$,
\begin{equation*}
  \be_0 (\pi_{\sigma_j}) = \frac{t! \, (s-t+c-j)!}{(s+c) \, (t+c-j)! \, (s-t+c-1)!}.
\end{equation*}
Thus, we have to show the identity
\begin{align*}
  1 = & \frac{(t+c-1)!}{t!} \left [(s+1-t) \left (
\frac{t!}{(s+c) \, (t+c-1)!} + \frac{(s-t)!}{(s+c) \, (s-t+c-1)!} \right ) \right. \\[1ex]
& \left. + \, (s+1-2t) \sum_{j=2}^{c-1} \frac{t! \, (s-t+c-j)!}{(s+c) \, (t+c-j)! \, (s-t+c-1)!} \right ] .
\end{align*}
It is equivalent to
\begin{align*}
\frac{s+c}{s+1-2t} &  =  \frac{s+1-t}{s+1-2t} \left [ 1 + \frac{(s-t)! \, (t+c-1)!}{(s-t+c-1)! \, t!} \right ] + \sum_{j=2}^{c-1} \frac{(t+c-1)! \, (s-t+c-j)!}{(t+c-j)! \, (s-t+c-1)!} \\[1ex]
& = \frac{s+1-t}{s+1-2t} \left [ 1 + \frac{\binom{t+c-1}{c-1}}{\binom{s-t+c-1}{c-1}} \right ] + \sum_{j=2}^{c-1} \frac{\binom{t+c-1}{j-1}}{\binom{s-t+c-1}{j-1}} \\[1ex]
& = \frac{s+1-t}{s+1-2t} \left [ 1 + \frac{\binom{t+c-1}{c-1}}{\binom{s-t+c-1}{c-1}} \right ] -  \frac{(s+1-t) \, \binom{t+c-1}{c-1}}{(s+1-2t) \, \binom{s-t+c-1}{c-1}} + \frac{t+c-1}{s+1-2t},
\end{align*}
which is certainly true and
where we used Lemma \ref{lem:Gor-identity} with $a = t+c-1$ and $b = s+1-2t$ to establish  the last equality. This completes the argument.
\end{proof}

\begin{rmk} Notice that the summands appearing in the Boij-S\"oderberg decomposition in Theorem \ref{thm:Gor-ideals} satisfy
\[
\pi_{\sigma_i} = \pi_{s+c + \sigma_{c+1-i}^*}.
\]
This is in accordance with Theorem \ref{thm:BS-self-dual}.
\end{rmk}

We illustrate the last result in the case of stacked polytopes. Recall that a $d$-dimensional simplicial polytope is stacked if it admits a triangulation $\Gamma$ which is a $(d-1)$-tree, that is, $\Gamma$ is a shellable $(d-1)$-dimensional simplicial complex with $h$-vector $(1, c-1)$. For example, such a polytope is obtained by pairwise gluing of $d$-simplices along a facet. The following result shows how the decomposition of the Betti table of its boundary complex reflects the data that determine the polytope.

\begin{cor}
   \label{cor:stacked-polytope}
Let $\Delta$ be the boundary complex of a stacked polytope with $n = c+d$ vertices  that is obtained by stacking $c$ simplices of dimension $d$.
Then the Betti table of its Stanley-Reisner ring $K[\Delta]$ has the   Boij-S\"oderberg decomposition
\begin{equation*}
   \label{eq:BS-decomp-stacked}
 \beta(K[\Delta])= d \cdot c! \cdot \left [\pi_{\sigma_1} + \pi_{\sigma_c} \right]  + (d-1) \cdot c!  \cdot \sum_{j=2}^{c-1} \pi_{\sigma_j},
\end{equation*}
where
\[
\sigma_j=(0,2,..,c-j+1,n-1-j,..,n-2,n).
\]
\end{cor}

\begin{proof}
The graded Betti numbers of $K[\Delta]$ are given by Lemma \ref{lem:Gor-res} with
$s=d = n-c$ and $t=1$. Hence Theorem  \ref{thm:Gor-ideals} yields the claim.
\end{proof}

\begin{rmk}
In the special case of 3-dimensional stacked polytopes ($d = 3$),  Corollary \ref{cor:stacked-polytope}  establishes a conjecture made in the report of group 10.1 (V. Kalyankar, V. Lorman, S. Seo, M. Stamps, Z. Yang) at the MSRI summer workshop in commutative algebra 2011.
\end{rmk}

We conclude by considering a specific instance of the last result.

\begin{eg}
   \label{ex:stacked}
Consider a 3-dimensional polytope on the seven vertices $a,\ldots,f,v$, obtained by stacking four 3-simplices with common vertex $v$.
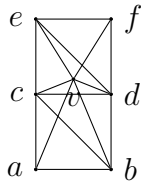
\begin{center}
\begin{tikzpicture}[scale=1, vertices/.style={draw, fill=black, circle, inner sep=0pt}]
    \node [vertices, label=left:{$a$}] (1) at (0,0){};
    \node [vertices, label=right:{$b$}] (2) at (1,0){};
    \node [vertices, label=left:{$c$}] (3) at (0,1){};
    \node [vertices, label=right:{$d$}] (4) at (1,1){};
    \node [vertices, label=left:{$e$}] (5) at (0,2){};
    \node [vertices, label=right:{$f$}] (6) at (1,2){};
    \node [vertices, label=below:{$v$}] (7) at (1/2,1.2){};
  \foreach \to/\from in {1/2,1/3,2/3,2/4,3/4,3/5,4/5,4/6,5/6, 7/1,7/2,7/3,7/4,7/5,7/6}
  \draw [-] (\to)--(\from);
\end{tikzpicture}
\end{center}
Denote by $\Delta$ its boundary complex. Its Stanley-Reisner ideal is
\[
I_{\Delta} = (ad,ae,af,be,bf,cf,bcv,cdv,dev).
\]
The Betti table of the Stanley-Reisner ring $K[\Delta]$ is
\[
\beta(K[\Delta])= \begin{array}{c|ccccc}
\beta_{i,j} & 0 & 1 & 2 & 3 & 4\\
\hline
0 & 1 & \cdot & \cdot & \cdot & \cdot\\
1 & \cdot & 6 & 8 & 3 & \cdot\\
2 & \cdot & 3 & 8 & 6 & \cdot\\
3 & \cdot & \cdot & \cdot & \cdot & 1
\end{array}
\]
It has the following Boij-S\"oderberg decomposition:
\begin{align*}
\beta(K[\Delta])= \; & 3\cdot4!\begin{array}{c|ccccc}
\beta_{i,j} & 0 & 1 & 2 & 3 & 4\\
\hline
0 & \frac{1}{168} & \cdot & \cdot & \cdot & \cdot\\
1 & \cdot & \frac{1}{20} & \frac{1}{12} & \frac{1}{24} & \cdot\\
2 & \cdot & \cdot & \cdot & \cdot & \cdot\\
3 & \cdot & \cdot & \cdot & \cdot & \frac{1}{420}
\end{array} +
2\cdot4!\begin{array}{c|ccccc}
\beta_{i,j} & 0 & 1 & 2 & 3 & 4\\
\hline
0 & \frac{1}{210} & \cdot & \cdot & \cdot & \cdot\\
1 & \cdot & \frac{1}{30} & \frac{1}{24} & \cdot & \cdot\\
2 & \cdot & \cdot & \cdot & \frac{1}{60} & \cdot\\
3 & \cdot & \cdot & \cdot & \cdot & \frac{1}{280}
\end{array} \\
& + 2\cdot4!\begin{array}{c|ccccc}
\beta_{i,j} & 0 & 1 & 2 & 3 & 4\\
\hline
0 & \frac{1}{280} & \cdot & \cdot & \cdot & \cdot\\
1 & \cdot & \frac{1}{60} & \cdot & \cdot & \cdot\\
2 & \cdot & \cdot & \frac{1}{24} & \frac{1}{30} & \cdot\\
3 & \cdot & \cdot & \cdot & \cdot & \frac{1}{210}
\end{array} +
3\cdot4!\begin{array}{c|ccccc}
\beta_{i,j} & 0 & 1 & 2 & 3 & 4\\
\hline
0 & \frac{1}{420} & \cdot & \cdot & \cdot & \cdot\\
1 & \cdot & \cdot & \cdot & \cdot & \cdot\\
2 & \cdot & \frac{1}{24} & \frac{1}{12} & \frac{1}{20} & \cdot\\
3 & \cdot & \cdot & \cdot & \cdot & \frac{1}{168}
\end{array}
\end{align*}
Notice how one can read off the dimension and number of stacked simplices in the polytope from the coefficients in the decomposition.
\end{eg}


\begin{thebibliography}{99}

\bibitem{AHH1} A.\ Aramova, J.\ Herzog, T.\ Hibi, {\em Shifting operations and graded Betti numbers}, J. Algebraic Combin. {\bf 12} (2000), no.\ 3, 207-–222.

\bibitem{AHH2} A.\ Aramova, J.\ Herzog, T.\ Hibi, {\em Ideals with stable Betti numbers}, Adv. Math. 152 (2000), no.\ 1, 72-–77.

\bibitem{B} C.\ B\"ohning, {\em Canonical surfaces $\PP^4$ with $p_g=p_a=5$ and $K^2=11$}, Atti Accad.\ Naz.\ Lincei Cl.\ Sci.\ Fis.\ Mat.\ Natur.\ Rend.\ Lincei (9) Mat. Appl.\ {\bf 18} (2007), 39–-57.


\bibitem{BH} W.\ Bruns, J.\ Herzog, {\em Cohen-Macaulay rings. Rev. ed.}, Cambridge Studies in Advanced Mathematics {\bf 39}, Cambridge University Press, Cambridge, 1998.

\bibitem{BS1} M.\ Boij, J.\ S\"oderberg, {\em  Graded Betti numbers of Cohen-Macaulay modules and the multiplicity conjecture}, J.\ London Math.\ Soc.\ {\bf 78} (2008), 85–-106.

\bibitem{BS2} M.\ Boij, J.\ S\"oderberg, {\em Betti numbers of graded modules and the multiplicity conjecture in the non-Cohen-Macaulay case}, Algebra Number Theory
(to appear); preprint available at  arXiv:0803.1645.

\bibitem{CN1}
    A.\ Corso, U.\ Nagel, {\em Monomial and toric ideals associated to Ferrers graphs}, Trans.\ Amer.\ Math.\ Soc.\ {\bf 361} (2009), 1371–-1395.

\bibitem{CN2}
    A.\ Corso, U.\ Nagel, {\em Specializations of Ferrers ideals}. J.\ Algebraic Combin.\ {\bf 28} (2008), 425–-437.

\bibitem{ES} D.\ Eisenbud, F.\ Schreyer, {\em Betti numbers of graded modules and cohomology of vector bundles}, J.\ Amer.\ Math.\ Soc.\ {\bf 22} (2009), 859–-888.

\bibitem{G}  M.\ Green, {\em Generic initial ideals}, in: {\em Six lectures on commutative algebra},  Progress in Mathematics
{\bf 166}, Birkh\"auser, 1998, pp.\ 119--185.

\bibitem{HH} J.\ Herzog, T.\ Hibi, {\em Monomial Ideals}, Graduate Texts in
    Mathematics {\bf 260}, Springer, 2011.

\bibitem{HSV} J.\ Herzog, L.\ Sharifan, M.\ Varbaro, {\em Graded Betti numbers of componentwise linear ideals}, Preprint, 2011; available at  arXiv:1111.0442.

\bibitem{MN}
    J.\ Migliore, U.\ Nagel, {\em Reduced arithmetically Gorenstein schemes and simplicial polytopes with maximal Betti numbers}, Adv.\ Math.\ {\bf 180} (2003), 1--63.

\bibitem{NR1}
    U.\ Nagel, V.\ Reiner, {\em Betti numbers of monomial ideals and shifted skew shapes}, Electron.\ J.\ Combin.\ {\bf 16} (2) (2009), Research Paper 3, 59 pp.

\bibitem{NR2}
    U.\ Nagel, T.\ R\"omer, {\em Criteria for componentwise linearity}, Preprint, 2011; available at  arXiv::1108.3921.

\bibitem{mod-liai}
    U.\ Nagel, {\em Liaison classes of modules}, J.\ Algebra {\bf 284} (2005), 236--272.

\end{thebibliography}
\end{document}